\documentclass[10pt,twoside]{amsart}
\usepackage{amssymb}

\usepackage{amsmath, amsthm, amscd, amsfonts, amssymb, graphicx, color}
\usepackage[bookmarksnumbered, colorlinks, plainpages]{hyperref}

\newtheorem{thm}{Theorem}[section]
\newtheorem{cor}[thm]{Corollary}
\newtheorem{lem}[thm]{Lemma}
\newtheorem{prob}[thm]{Problem}

\newtheorem{obs}[thm]{Observation}

\newtheorem{prop}[thm]{Proposition}
\newtheorem{defn}[thm]{Definition}
\newtheorem{rem}[thm]{\bf{Remark}}

\numberwithin{equation}{section}

\begin{document}



\title{Total Dominator Chromatic Number of a Graph}
\author{Adel P. Kazemi\\
College of Mathematical Sciences\\
University of Mohaghegh Ardabili\\
P.O.Box 5619911367, Ardabil, Iran\\
Email: adelpkazemi@yahoo.com
}

\thanks{{\scriptsize
\hskip -0.4 true cm MSC(2010): 
05C15; 05C69.
\newline Keywords: Total dominator chromatic number, total domination number, chromatic number.\\
}}

\maketitle

\begin{abstract}
Given a graph $G$, the total dominator coloring problem seeks a
proper coloring of $G$ with the additional property that every
vertex in the graph is adjacent to all vertices of a color class. We
seek to minimize the number of color classes. We study this problem
on several classes of graphs, as well as finding general bounds and
characterizations. We also show the relation between total dominator
chromatic number and chromatic number and total domination number.
\end{abstract}
\vskip 0.2 true cm


\pagestyle{myheadings}
\markboth{\rightline {\scriptsize  A. P. Kazemi}}
         {\leftline{\scriptsize Total dominator chromatic number in graphs}}

\bigskip
\bigskip


\section{\bf Introduction}
\vskip 0.4 true cm

All graphs considered here are finite, undirected and simple. For
standard graph theory terminology not given here we refer to \cite{West}. Let $%
G=(V,E) $ be a graph with the \emph{vertex set} $V$ of \emph{order}
$n(G)$ and the \emph{edge set} $E$ of \emph{size} $m(G)$. The
\emph{open neighborhood} and the \emph{closed neighborhood} of a
vertex $v\in V$ are $N_{G}(v)=\{u\in V\ |\ uv\in E\}$ and
$N_{G}[v]=N_{G}(v)\cup \{v\}$, respectively. The \emph{degree} of a
vertex $v$ is also $deg_G(v)=\mid N_{G}(v) \mid $. The
\emph{minimum} and \emph{maximum degree} of $G$ are denoted by
$\delta =\delta (G)$ and $\Delta =\Delta (G)$, respectively. If $\delta (G)=\Delta (G)=k$, then $G$ is called $k$-\emph{regular}. We say
that a graph is \emph{connected} if there is a path between every
two vertices of the graph, and otherwise is called
\emph{disconnected}. We write $K_{n}$, $C_{n}$ and $P_{n}$ for a \emph{complete graph}, a \emph{cycle} and a
\emph{path} of order $n$, respectively, while $G[S]$ and
$K_{n_1,n_2,...,n_p}$ denote the \emph{subgraph induced} of $G$ by a
vertex set $S$ of $G$ and the \emph{complete $p$-partite graph},
respectively. The \emph{complement} of a graph $G$ is denoted by $\overline{G}$
and is a graph with the vertex set $V(G)$ and for every two vertices
$v$ and $w$, $vw\in E(\overline{G})$\ if and only if $vw\not\in
E(G)$.

A \emph{total dominating set} (resp. \emph{dominating set}) $S$ of a graph $G$ is a subset of the vertices in $G$ such that for each vertex $v$, $N_G(v)\cap S\neq \emptyset$ (resp. $N_G[v]\cap S\neq \emptyset$). The
\emph{total domination number} $\gamma_t(G)$ (resp. \emph{domination number $\gamma(G)$}) of $G$ is the cardinality of a
minimum total dominating set (resp. dominating set). The topics has long been of interest to
researchers \cite{hhs1, hhs2}.

A \emph{proper coloring} of a graph $G =(V,E)$ is a function from
the vertices of the graph to a set of colors such that any two
adjacent vertices have different colors. The \emph{chromatic number}
$\chi (G)$ of $G$ is the minimum number of colors needed in a proper
coloring of a graph. In a proper coloring of a graph a \emph{color class} is the set of all same colored vertices of the graph.
Graph coloring is used as a model for a vast
number of practical problems involving allocation of scarce
resources (e.g., scheduling problems), and has played a key role in
the development of graph theory and, more generally, discrete
mathematics and combinatorial optimization. Graph $k$-colorability
is NP-complete in the general case, although the problem is solvable
in polynomial time for many classes \cite{GJ}.

A \emph{dominator coloring} of a graph $G$, briefly DC, is a proper coloring of $G$ such that every vertex of $V(G)$ dominates all
vertices of at least one color class (possibly its own class). The \emph{dominator chromatic
number} $\chi_d(G)$ of $G$ is the minimum number of color classes in a dominator coloring of $G$. As a consequence result we have $\chi (G) \leq \chi_d(G)$. The concept of dominator coloring was introduced recently by Gera et al. \cite{GHR} and studied
further in \cite{MM,Ger1,Ger2}. Here, we initiate to the study of a similar concept, total dominator coloring, in graphs.
\begin{defn}
\label{total dominator coloring} \emph{ A} total dominator coloring \emph{of a graph $G$, briefly TDC, is a proper
coloring of $G$ in which each vertex of the graph is adjacent to every
vertex of some color class. The }total dominator chromatic
number $\chi_d^t(G)$ \emph{of $G$ is the minimum number of color classes in a
total dominator coloring of $G$. A }$\chi_d^t(G)$-coloring \emph{of
$G$ is any total dominator coloring with $\chi_d^t(G)$ colors}.
\end{defn}
If $f$ is a total dominator coloring or a proper coloring of $G$
with the coloring classes $V_1$, $V_2$, ..., $V_{\ell}$ such that every vertex in $V_i$ has color $i$, we write
simply $f=(V_1,V_2,...,V_{\ell})$. In the following two definitions $f=(V_1,V_2,...,V_{\ell})$ is a total dominator coloring of $G$.
\begin{defn}
\label{common neighborhood}
\emph{A vertex $v$ is called a }common neighbor \emph{of $V_i$ if $v\succ V_i$, that is, $v$ is adjacent to all vertices in $V_i$. The set of all common neighbors of $V_i$ is called the }common neighborhood \emph{of $V_i$ in $G$ and denoted by $CN_G(V_i)$ or simply $CN(V_i)$.}
\end{defn}
\begin{defn}
\label{private neighborhood}
\emph{A vertex $v$ is called the} private neighbor of $V_i$ with respect to \emph{$f$ if $v\succ V_i$ and $v\nsucc V_j$ for all $j\neq i$. The set of all private neighbors of $V_i$ is called the} private neighborhood \emph{of $V_i$ in $G$ and denoted by $pn_G(V_i;f)$ or simply $pn(V_i;f)$.}
\end{defn}

The following proposition can be easily proved by Definitions \ref{total dominator coloring} and \ref{common neighborhood}.
\begin{prop}
\label{U(CN_(V_i))=V(G)} Let $f=(V_1,V_2,...,V_{\ell})$ be a total dominator coloring of $G$, and let $I=\{i\mbox{ }|\mbox{ } |V_i|\leq \Delta(G)\}$. Then $V(G)=\cup_{i\in I}CN_G(V_i)$.
\end{prop}

In this paper, we study the total dominator chromatic
number on several classes of graphs, as well as finding general bounds and
characterizations. We show also its relationship with chromatic number and total domination number.

The next known result is useful for our investigations.

\begin{prop}
\label{gmmma_t(G square H)=<min(gamma_t G.n,gamma_t H.m)} \emph{(\textbf{Kazemi, Pahlavsay \cite{Kaz.Pah} 2012)}} Let $G$ and $H$ be two graphs without isolated vertices. Then $\gamma _t(G\Box H)\leq \min\{\gamma _t(G)|V(H)|,\gamma _t(H) |V(G)|\}$.
\end{prop}

\section{\bf Complexity}
\vskip 0.4 true cm

In this section we formally establish the difficulty of finding the total dominator coloring number of an arbitrary graph. First we define some relevant decision problems.\\

\textsc{chromatic number} Given a graph $G$ and a positive integer $k$, does there exist a function
$f : V(G)\rightarrow \{1, 2, ..., k\}$ such that $f(u)\neq f(v)$ whenever $uv\in E(G)$?\\

\textsc{total dominator chromatic number} Given a graph $G$ and a positive integer $k$, does there
exist a function $f : V(G)\rightarrow \{1, 2, ..., k\}$ such that $f(u)\neq f(v)$ whenever $uv\in E(G)$ and
for any vertex $x\in V(G)$ there exists a color $i$ such that $\{y\in V(G) | f(y) = i\}\subseteq N(x)$?

\begin{thm}
\label{Total dominator chromatic number is NP-complete} \textsc{total dominator chromatic number} is NP-complete.
\end{thm}

\begin {proof}
\textsc{total dominator chromatic number} is clearly in NP, since we can efficiently verify
that an assignment of colors to the vertices of $G$ is both a proper coloring and that every
vertex $v$ dominates some color class other than the color class of $v$.

Now we transform \textsc{chromatic number} to \textsc{total dominator chromatic number}. Consider
an arbitrary instance $(G,k)$ of \textsc{chromatic number}. Create an instance $(G',k')$ of
\textsc{total dominator chromatic number} as follows. Add a vertex $v'$ to $G$ and add an edge from
$v'$ to every vertex in $G$. Set $k' \rightarrow k+1$.

Suppose $G$ has a proper coloring using $k$ colors. Then the coloring of $G'$ that colors $v'$
with a new color is a proper coloring of $G'$. Since $v'\in N(u)$ for every $u\in V(G)$ and $\{u\in V(G) | f(u) = i\}\subseteq N(v')$ for some color $i$ (more exactly, for all colors $i$) other than the color of $v$, this coloring is a total dominator coloring, and is uses $k'=k+1$ colors.

Now suppose $G'$ has a total dominator coloring using $k'$ colors. Since $v'$ is adjacent to every
other vertex in $G'$, it must be the only vertex of its color in the hypothesized coloring. Then
the removal of $v'$ leaves a proper coloring of $G$ that uses $k'-1=k$ colors.
\end{proof}

\section{\bf Some bounds}
\vskip 0.4 true cm

In this section we will present some sharp lower and upper bounds for the total dominator chromatic number of a graph. First, we state the following observation.

\begin{obs}
\label{obs} Let $G$ be a graph of order $n$ and without isolated vertices. Then
\[
\max\{ \chi_d(G),\gamma_t(G)\} \leq \chi_d^t(G) \leq n.
\]
\end{obs}

The next theorem gives some lower and upper bounds for the total dominator chromatic number of a graph in terms of the total dominator chromatic numbers of its connected components.
\begin{thm}
\label{chi_d ^t,components} Let $G$ be a graph without isolated vertices. If $G_1$, $G_2$, ..., $G_{\omega}$ are all connected components of $G$, then
\[
\max_{1\leq i \leq \omega} \chi_d ^t(G_i) +2\omega -2 \leq \chi_d ^t(G) \leq \Sigma _{i=1}^{\omega} \chi_d ^t(G_i).
\]
\end{thm}

\begin{proof}
For $1\leq i \leq \omega$, let $f_i$ be a $\chi_d ^t$-coloring of $G_i$. Let $f$ be a function on $V(G)$ such that for any vertex $v\in V(G_i)$, $f(v)=(i,f_i(v))$. Then $f$ is a total dominating coloring of $G$, and so
$\chi_d ^t(G) \leq \Sigma _{i=1}^{\omega} \chi_d ^t(G_i)$.

Now let $\chi_d ^t(G_j)=\max_{1\leq i \leq \omega} \chi_d ^t(G_i)$, for some $1\leq j\leq \omega$.
Since we need to at least two new colors for coloring the vertices of every $G_i$, when $i\neq j$, we obtain
\[
\chi_d ^t(G) \geq \max_{1\leq i \leq \omega} \chi_d ^t(G_i) +2\omega -2.
\]
\end{proof}

In Theorem \ref{chi_d ^t,components}, we trivially see that
\[
\chi_d ^t(G)=\max_{1\leq i \leq \omega} \chi_d ^t(G_i) +2\omega -2
\]
if and only if at most one connected component of $G$ is not complete bipartite graph. Therefore, in continuation to our discussion, we assume that $G$ is a connected graph.

Next theorem present the lower bound 2 and the upper bound $n$ for the total dominator chromatic number of a connected graph of order $n$ which has no isolated vertex.

\begin{thm}
\label{2=<chi_d ^t=<n} If $G$ is a connected graph of order $n$ and without isolated vertices,
then $2\leq \chi_d ^t(G)\leq n$. Furthermore, $\chi_d ^t(G)$ is 2 or $n$ if and only if $G$ is a
complete bipartite graph, or is isomorphic to the complete graph $K_n$, respectively.
\end{thm}

\begin{proof}
Observation \ref{obs} implies $\chi_d ^t(G)\geq
\gamma_t(G)$, and since the total domination number of any graph
is at least 2, we obtain $2\leq \chi_d ^t(G)\leq n$.

If $G$ is a complete bipartite graph or is isomorphic to the complete graph $K_n$, then, obviously, $\chi_d
^t(G)=2$ or $\chi_d^t(G)=n$, respectively. Now let $\chi_d ^t(G)=2$, and let $f:V(G)\rightarrow
\{1,2\}$ be a $\chi_d ^t(G)$-coloring. If $V_i=\{v\in V(G) \mid
f(v)=i\}$, for $i=1,2$, then $G$ is the complete bipartite graph
with the vertex partition $V(G)=V_1\cup V_2$.

In the second case, we assume that $G$ is not isomorphic to the complete graph $K_n$, and $\chi_d^t(G)=n$.
Let $f$ be a $\chi_d^t(G)$-coloring. Without loss of generality, we may assume that $n\geq 3$. If $deg_G(x)=1$, for some vertex $x$, then by choosing $\alpha $ as an arbitrary element in $\{1,2,3,...,n\}-\{f(x)\}$, for each vertex $v$ we define
\begin{equation*}
g(v)=\left\{
\begin{array}{ll}
f(v) & \mbox{if }v\neq x, \\
\alpha & \mbox{if }v=x.
\end{array}
\right.
\end{equation*}
Thus $g$ is a total dominator coloring of $G$ with $n-1$ color classes, and so $\chi_d^t(G)<n$, a contradiction. Therefore, we may assume $\delta (G)\geq 2$. Now let $u$ and $u'$ be two non-adjacent vertices in $G$. Then the function $h$ on $V(G)$ with definition
\begin{equation*}
h(v)=\left\{
\begin{array}{ll}
f(v) & \mbox{if }v\neq u, \\
f(u') & \mbox{if }v=u,
\end{array}
\right.
\end{equation*}
is a total dominator coloring of $G$ with $n-1$ color classes, and so $\chi_d^t(G)<n$, a contradiction.
Therefore, $G$ is isomorphic to the complete graph $K_n$.
\end{proof}

Let $S$ be an independent vertex set in a graph $G=(V,E)$ such that the induced subgraph $G[V-S]$ has no isolated vertex or every isolated vertex in it is adjacent to all vertices in $S$. Let $\alpha_0(G)$ be the maximum cardinality of such a set in $G$. With this definition and notation we state following.

\begin{thm}
\label{chi_d^t(G)=<n+1-alpha_0} Let $G$ be a connected graph of order $n$ and without isolated vertices. Then
\[
\chi_d^t(G) \leq n+1- \alpha_0(G).
\]
\end{thm}

\begin{proof}
Let $S$ be an independent vertex set in $G$ such that the induced subgraph $G[V(G)-S]$ has no isolated vertex or every isolated vertex in it is adjacent to all vertices of $S$ and $\mid S \mid=\alpha_0(G)$. We assign $n-\alpha_0(G)$ colors to $n-\alpha_0(G)$ vertices in $G[V(G)-S]$, and then assign ($n-\alpha_0(G)+1$)-th color to all vertices in $S$. This is a total dominator coloring of $G$, and so $\chi_d^t(G) \leq n+1- \alpha_0(G)$.
\end{proof}

\begin{cor}
\label{chi_d^t(G) =<n+1- alpha(G)} Let $G$ be a connected $k$-regular graph of order $n$ and without isolated vertices. If $\alpha(G)=k$, then
\[
\chi_d^t(G) \leq n+1- \alpha(G).
\]
\end{cor}
Next theorem present a sharp upper bound for the total dominator chromatic number of a connected graph in terms of its total domination number and the chromatic number of an induced subgraph of it.
\begin{thm}
\label{chi_d^t =<g_t+min{chi(G-S)}} Let $G$ be a connected graph without
isolated vertices. Then
\[
\chi_d^t(G) \leq \gamma_t(G)+ \min_S \chi(G[V(G)-S]),
\]
where $S\subseteq V(G)$ is a $\gamma_t(G)$-set. Also this upper bound is sharp.
\end{thm}

\begin{proof}
Let $\ell=\min\{\chi (G[V(G)-S]) \mid \mbox{ } S \mbox{ is a }
\gamma_t(G) \mbox{-set}\}$, and let $D=\{v_1,v_2,...,v_m\}$
be a $\gamma_t(G)$-set such that $\chi(G[V(G)-D])=\ell$. Let also
$f: V(G)-D \rightarrow \{1,2,...,\ell\}$ be a proper coloring of
$G[V(G)-D]$. We define $g: V(G) \rightarrow \{1,2,3,...,\ell+m\}$ such that
\begin{equation*}
g(v)=\left\{
\begin{array}{ll}
\ell+i & \mbox{if }v=v_i \in D, \\
f(v) & \mbox{if }v\not \in D.
\end{array}
\right.
\end{equation*}
Since $D$ is a total dominating set of $G$, $g$ will be a total dominator coloring of $G$. Hence
\[
\chi_d^t(G) \leq m + \ell = \gamma_t(G)+\min\{\chi
(G[V(G)-S]) \mid \mbox{ } S \mbox{ is a } \gamma_t(G) \mbox{-set}
\}.
\]

This upper bound is sharp. For example, if $K_n$ is the complete graph of
order $n\geq 3$, then
\begin{equation*}
\begin{array}{lll}
\chi_d^t(K_n)& = & n \\
& = & \gamma _{t}(K_n)+ \chi_d^t(K_{n-2}) \\
& = & \gamma _{t}(K_n)+ \min\{\chi (K_n[V-S]) \mid \mbox{ } S \mbox{
is a } \gamma_t(K_n) \mbox{-set} \}.
\end{array}%
\end{equation*}%
Also it can be verified that this bound is sharp for the complete $p$-partite graph $K_{1,1,n_1,...,n_{p-2}}$, where $p\geq 3$, and for any wheel $W_n$, where $n\geq 3$ is odd (see Proposition \ref{chi_dt W_n}).
\end{proof}

\begin{cor}
\label{chi_d^t =<g_t+p} If $G$ is a connected $p$-partite graph without
isolated vertices, then
\[
\chi_d^t(G) \leq \gamma_t(G)+ p.
\]
\end{cor}

The next result gives another upper bound for a connected $p$-partite graph.

\begin{thm}
\label{chi_d ^t(p-partite)=<n-n'+1} Let $G$ be a connected $p$-partite graph of order $n$. Let $n_1$, $n_2$, ..., $n_p$ be the cardinality of the $p$-partite sets of $G$. If $\delta(G)\geq n_i$, for some $i$, then $\chi_d^t(G)\leq n-n'+1$, where $n'=\max\{n_i|\delta(G)\geq n_i\}$.
\end{thm}

\begin{proof}
Let $G$ be a connected $p$-partite graph of order $n$ with $V_1$, ..., $V_p$ as $p$ independent sets of $V(G)$ such that $|V_j|=n_j$, for $1 \leq j \leq p$. Let $n'=n_i$, for some $i$. Then the coloring that assigns colors 1, 2, ..., $n-n_i$ to the vertices of $V(G)-V_i$, and color $n-n_i+1$ to the vertices of $V_i$, is a TDC of $G$. Hence $\chi_d^t(G)\leq n-n'+1$.
\end{proof}

We notice that if a graph $G$ has a $\chi_d^t$-coloring $f$ without singleton color class, then $f$ is also a dominator coloring of $G$, and hence $\chi_d^t(G)=\chi_d(G)$. Next proposition shows that this condition is not necessary for $\chi_d^t(G)=\chi_d(G)$.

\begin{prop}
\label{chi_d^t =chi, Delta=n-1} Let $G$ be a connected graph of order $n$ and without isolated
vertices. If $\Delta (G)=n-1$, then $\chi_d^t(G)=\chi_d(G)=\chi(G)$.
\end{prop}

\begin{proof}
Let $f=(V_1,V_2,...,V_m)$ be a proper coloring of $G$, where
$m=\chi (G)$, and $V_1=\{v\}$ for some vertex $v$ of degree $n-1$. Then $w\succ V_1$ for each
vertex $w\in V(G)-V_1$. Also for each $2\leq i \leq m$, $v\succ
V_i$. Therefore $f$ is a total dominator coloring of $G$ with $\chi(G)$ color classes, and so $\chi_d^t(G)\leq \chi(G)$. Now
Observation \ref{obs} implies $\chi_d^t(G)=\chi_d(G)=\chi(G)$.
\end{proof}

\begin{cor}
\label{chi_d^t =chi, Delta=n-1, ell} Let $G$ be a connected graph of order $n$ and without
isolated vertices. If $\Delta(G)=n-1$ and $v_1$, ..., $v_{\ell}$ be all vertices of degree $n-1$, then
\[
\chi_d^t(G)=\ell+\chi(G[V-\{v_1,...,v_{\ell}\}]).
\]
\end{cor}

\section{\bf The total dominator chromatic number of some graphs}
\vskip 0.4 true cm

Obviously, the total dominator chromatic number of every complete
$p$-partite graph is $p$. In this section we calculate this number
for some other classes of graphs.
\begin{prop}
\label{chi_dt W_n} Let $W_n$ be a wheel of order $n+1\geq 4$. Then
\begin{equation*}
\chi_d^t(W_n)=\left\{
\begin{array}{ll}
3 & \mbox{if }n \mbox{ is even}, \\
4 & \mbox{if }n \mbox{ is odd}.
\end{array}
\right.
\end{equation*}
\end{prop}

\begin{proof}
As a consequence of Corollary \ref{chi_d^t =chi, Delta=n-1, ell}, we have
\begin{equation*}
\begin{array}{lll}
\chi_d^t(W_n)& = 1+\chi(C_n) & \\
&=\left\{
\begin{array}{ll}
3 & \mbox{if }n \mbox{ is even}, \\
4 & \mbox{if }n \mbox{ is odd}.
\end{array}
\right.
\end{array}%
\end{equation*}%
\end{proof}

Notice that $\chi_d^t(W_n)=\chi_d(W_n)$, by \cite{Ger1}.


\begin{prop}
\label{chi_dt C_n} Let $C_n$ be a cycle of order $n\geq 3$. Then
\begin{equation*}
\chi_d^t(C_n)=\left\{
\begin{array}{ll}
2                                & \mbox{if }n=4, \\
4\lfloor \frac{n}{6}\rfloor +r   & \mbox{if } n\neq 4 \mbox{ and for }r=0,1,2,4,~~n\equiv r\pmod{6},\\
4\lfloor \frac{n}{6}\rfloor +r-1 & \mbox{if }~~n\equiv r\pmod{6} \mbox{, where }r=3,5.
\end{array}
\right.
\end{equation*}
\end{prop}

\begin{proof}
Let $V(C_n)=\{v_i\mbox{ }|\mbox{ }1\leq i \leq n\}$, and let $v_iv_j\in E(C_n)$ if and only if $|i-j|=1$ (to modulo $n$).
We claim that for every TDC $f$ of $C_n$, we need to at least four colors to color every six consecutive vertices $v_i$, $v_{i+1}$, $v_{i+2}$, $v_{i+3}$, $v_{i+4}$ and $v_{i+5}$. Trivially, we may assume that some color, say $a$, appear at least two times. We assign colors $a$, $b$, $a$ to vertices $v_i$, $v_{i+1}$, $v_{i+2}$, respectively. We can assign color $b$ to vertex $v_{i+3}$ or not. In each case, we need to at least two new colors $c$ and $d$ for coloring the remained vertices. Because, in the first case, we have to assign two new colors $c$ and $d$ to the vertices $v_{i+4}$ and $v_{i+5}$, respectively, and in the second case, we must assign colors $c$, $d$, $c$ to the vertices $v_{i+3}$, $v_{i+4}$, $v_{i+5}$, respectively. Therefore, our claim is proved. We also notice that any six consecutive vertices can be colored by four new colors $a$, $b$, $c$, $d$ in
\[
\emph{way 1: a,b,a,b,c,d},~~~ \mbox{   or   } ~~~\emph{way 2: a,b,a,c,d,c}.
\]
In \emph{way} 1, we have: $v_{i+1}\in pn(V_a;f)$, $v_{i+2}\in pn(V_b;f)$, $v_{i+3}\in pn(V_c;f)$, $v_{i+4}\in pn(V_d;f)$, while in \emph{way} 2 we have: $v_{i+1}\in pn(V_a;f)$, $v_{i+2}\in pn(V_b;f)$, $v_{i+4}\in pn(V_c;f)$, $v_{i+3}\in pn(V_d;f)$. We continue our proof in the following six cases.

\textbf{Case 0: } $n\equiv 0\pmod{6}$. In this case, if $f_0$ is a proper coloring which is obtained by each of ways 1 or 2 or by combining of them, then $f_0$ will be a TDC of $C_n$ with the minimum number $4\lfloor \frac{n}{6}\rfloor$ color classes, as desired.

\textbf{Case 1: } $n\equiv 1\pmod{6}$. In this case, let $f_0$ be the TDC of $C_n-\{v_n\}$ mentioned in Case 0. Since we need to one new color for coloring $v_n$, by assigning a new color $\varepsilon$ to $v_n$ we obtain a TDC of $C_n$ with the minimum number $4\lfloor \frac{n}{6}\rfloor +1$ color classes, as desired.

\textbf{Case 2: } $n\equiv 2\pmod{6}$. In this case, let $f_0$ be the TDC of $C_n-\{v_{n-1},v_n\}$ mentioned in Case 0. Since we need to two new colors for coloring $v_{n-1}$ and $v_n$, by assigning two new colors $\theta$, $\varepsilon$ to $v_{n-1}$, $v_n$, respectively, we obtain a TDC of $C_n$ with the minimum number $4\lfloor \frac{n}{6}\rfloor +2$ color classes, as desired.

\textbf{Case 3: } $n\equiv 3\pmod{6}$. In this case, let $f_0$ be the TDC of $C_n-\{v_{n-2},v_{n-1},v_n\}$ mentioned in Case 0. Since we need to two new colors for coloring $v_{n-2}$, $v_{n-1}$ and $v_n$, by assigning new colors $\varepsilon$, $\theta$, $\varepsilon$ to $v_{n-2}$, $v_{n-1}$, $v_n$, respectively, we obtain a TDC of $C_n$ with the minimum number $4\lfloor \frac{n}{6}\rfloor +2$ color classes, as desired.

\textbf{Case 4: } $n\equiv 4\pmod{6}$. In this case, let $f_0$ be the TDC of $C_n-\{v_{n-3},v_{n-2},v_{n-1},v_n\}$ mentioned in Case 0. Since we need to four new colors for coloring $v_{n-3}$, $v_{n-2}$, $v_{n-1}$ and $v_n$, by assigning new four colors $\pi$, $\varsigma$, $\theta$, $\varepsilon$ to $v_{n-3}$, $v_{n-2}$, $v_{n-1}$, $v_n$, respectively, we obtain a TDC of $C_n$ with the minimum number $4\lfloor \frac{n}{6}\rfloor +4$ color classes, as desired.

\textbf{Case 5: } $n\equiv 5\pmod{6}$. In this case, let $f_0$ be the TDC of $C_n-\{v_{n-4},v_{n-3},v_{n-2},v_{n-1},v_n\}$ mentioned in Case 0. Since we need to four new colors for coloring $v_{n-4}$, $v_{n-3}$, $v_{n-2}$, $v_{n-1}$, $v_n$, by assigning new colors $\pi$, $\varsigma$, $\pi$, $\theta$, $\varepsilon$ to the vertices $v_{n-4}$, $v_{n-3}$, $v_{n-2}$, $v_{n-1}$ and $v_n$, respectively, we obtain a TDC of $C_n$ with the minimum number $4\lfloor \frac{n}{6}\rfloor +4$ color classes, as desired.
\end{proof}

\begin{prop}
\label{chi_dt P_n} Let $P_n$ be a path of order $n\geq 2$. Then
\begin{equation*}
\chi_d^t(P_n)=\left\{
\begin{array}{ll}
2\lceil \frac{n}{3}\rceil -1   & \mbox{if } n\equiv 1\pmod{3},\\
2\lceil \frac{n}{3}\rceil & \mbox{otherwise}.
\end{array}
\right.
\end{equation*}
\end{prop}

\begin{proof}
Let $V(P_n)=\{v_i\mbox{ }|\mbox{ }1\leq i \leq n\}$ and for $1\leq i<j\leq n$, $v_iv_j\in E(C_n)$ if and only if $j=i+1$.
Let $f=(V_1,V_2,...,V_{\ell})$ be an arbitrary TDC of $P_n$. We see that any three, four or five consecutive vertices must be colored by at least two, three or four different colors, respectively. Because any vertex $v_i$ has degree two if $1<i<n$ and has degree one, otherwise. Therefore either $V_j=\{v_{i-1},v_{i+1}\}$ for some $1\leq j \leq \ell$, or $v_{i-1}\in V_j$ and $v_{i+1}\in V_k$ for some $1\leq j<k \leq \ell$ such that $|V_j|=1$ or $|V_k|=1$. This implies that $V(P_n)$ has partitioned to subsets of three consecutive vertices with colors $a,b,a$, or to subsets of four consecutive vertices with colors $a,b,c,a$, or to subsets of five consecutive vertices with colors either $a,b,a,c,d$, or $a,b,c,d,a$ (notice that the colors used in any part are different). By the previous discussion, it can be easily verified that the coloring function $f_0$ with
\begin{equation*}
f_0(v_i)=\left\{
\begin{array}{ll}
1+2k & \mbox{if }i=1+3k \mbox{ or } i=3+3k, \\
2+2k & \mbox{if }i=2+3k, \\
\end{array}
\right.
\end{equation*}
when $0 \leq k \leq \frac{n}{3}-1$, is a TDC of $P_n$ with the minimum number $2\lceil \frac{n}{3} \rceil$ color classes, if $n\equiv 0\pmod{3}$, as desired. Also, the coloring function $f_1$ with
\begin{equation*}
f_1(v_i)=\left\{
\begin{array}{ll}
1+2k & \mbox{if }i=1+3k \mbox{ or } i=3+3k, \\
2+2k & \mbox{if }i=2+3k, \\
\end{array}
\right.
\end{equation*}
when $0 \leq k \leq \lfloor \frac{n}{3} \rfloor -2$, and $f_1(v_{n-3})=f_1(v_n)=2\lfloor \frac{n}{3} \rfloor -1$, $f_1(v_{n-2})=2\lfloor \frac{n}{3} \rfloor$, $f_1(v_{n-1})=2\lfloor \frac{n}{3} \rfloor +1$, is a TDC of $P_n$ with the minimum number $2\lceil \frac{n}{3} \rceil -1$ color classes, if $n\equiv 1\pmod{3}$, as desired. Now let $n\equiv 2\pmod{3}$. If $n=2$, then $P_2=K_2$, and $\chi_d^t(P_2)=2$. Let $n=5$. In this case, $v_1$, $v_2$, $v_3$, $v_4$, $v_5$ can be colored in one of the ways: $a,b,a,c,d$, or $a,b,c,d,a$. Hence $\chi_d^t(P_5)=4$. Now let $n\geq 8$. Then the coloring function $f_2$ with
\begin{equation*}
f_2(v_i)=\left\{
\begin{array}{ll}
1+2k & \mbox{if }i=1+3k \mbox{ or } i=3+3k, \\
2+2k & \mbox{if }i=2+3k, \\
\end{array}
\right.
\end{equation*}
when $0 \leq k \leq \lfloor \frac{n}{3} \rfloor -2$, and $f_2(v_{n-4})=f_2(v_n)=2\lfloor \frac{n}{3} \rfloor -1$, $f_2(v_{n-3})=2\lfloor \frac{n}{3} \rfloor$, $f_2(v_{n-2})=2\lfloor \frac{n}{3} \rfloor +1$, $f_2(v_{n-1})=2\lfloor \frac{n}{3} \rfloor +2$, is a TDC of $P_n$ with the minimum number $2\lceil \frac{n}{3} \rceil$ color classes, as desired.
\end{proof}

\begin{prop}
\label{chi_d^t Cn Complement} Let $\overline{C_n}$ be the complement of the cycle $C_n$ of order $n\geq 4$. Then
\begin{equation*}
\chi_d^t(\overline{C_n})=\left\{
\begin{array}{ll}
4                                & \mbox{if }n=4,5, \\
\lceil \frac{n}{2}\rceil & \mbox{if }n\geq 6.
\end{array}
\right.
\end{equation*}
\end{prop}

\begin{proof}
let $V(\overline{C_n})=\{v_i|1\leq i \leq n\}$ and let $v_iv_j$ be an edge if and only if $j\neq i-1,i+1$. If $n=4,5$, then $\overline{C_n}$ is isomorphic to $2K_2$ or $C_5$, respectively, and thus $\chi_d^t(\overline{C_n})=4$. Now let $n\geq 6$. Since $\alpha(\overline{C_n})=2$, for any TDC $f=(V_1,V_2,...,V_{\ell})$ we have $|V_i|\leq 2$ for all $i$. Hence $\chi_d^t(\overline{C_n})\geq \lceil \frac{n}{2}\rceil$. Now for $1\leq i \leq \lfloor\frac{n}{2}\rfloor$ let $V_i=\{v_{2i},v_{2i-1}\}$. Then for even $n$, $f=(V_1,V_2,...,V_{\lfloor\frac{n}{2}\rfloor})$ is a TDC of $\overline{C_n}$ with $\lceil \frac{n}{2}\rceil$ color classes, while for odd $n$, $g=(V_1,V_2,...,V_{\lfloor\frac{n}{2}\rfloor},\{v_n\})$ is a TDC of $\overline{C_n}$ with $\lceil \frac{n}{2}\rceil$ color classes. Thus $\chi_d^t(\overline{C_n})=\lceil \frac{n}{2}\rceil$.
\end{proof}

\begin{prop}
\label{chi_d^t Pn Complement} Let $\overline{P_n}$ be the complement of the path $P_n$ of order $n\geq 4$. Then
\begin{equation*}
\chi_d^t(\overline{P_n})=\left\{
\begin{array}{ll}
3                                & \mbox{if }n=4, \\
\lceil \frac{n}{2}\rceil & \mbox{if }n\geq 5.
\end{array}
\right.
\end{equation*}
\end{prop}

\begin{proof}
let $V(\overline{P_n})=\{v_i|1\leq i \leq n\}$ and let $v_iv_j$ be an edge if and only if $\{i,j\}=\{1,n\}$ or $j\neq i-1,i+1$. Since $\overline{P_4}=P_4$, it is clear that $\chi_d^t(\overline{P_n})=3$. Now let $n\geq 5$. $\alpha(\overline{P_n})=2$ implies $\chi_d^t(\overline{P_n})\geq \lceil \frac{n}{2}\rceil$. Since also, the total dominator colorings given in Proposition \ref{chi_d^t Cn Complement} are also total dominator colorings of $\overline{P_n}$ with $\lceil \frac{n}{2}\rceil$ color classes, we obtain $\chi_d^t(\overline{P_n})=\lceil \frac{n}{2}\rceil$.
\end{proof}

\section{\bf A remark}

By comparing the propositions given in Section 4, we will obtain the following results.

\begin{prop}
\label{chi_d^t Pn, Cn} For any $n\geq 3$,
\begin{equation*}
\chi_d^t(P_n)=\left\{
\begin{array}{ll}
\chi_d^t(C_n)+1      & \mbox{if }n=4, \\
\chi_d^t(C_n)-1      & \mbox{if }n\equiv 4\pmod{6}\mbox{ and }n>4, \\
\chi_d^t(C_n)        & \mbox{otherwise.}
\end{array}
\right.
\end{equation*}
\end{prop}

\begin{prop}
\label{chi_d^t Wn, Cn} For any $n\geq 3$,
\begin{equation*}
\begin{array}{ll}
\chi_d^t(C_n)<\chi_d^t(W_n)      & \mbox{if }n=3,4, \\
\chi_d^t(C_n)=\chi_d^t(W_n)      & \mbox{if }n=5, \\
\chi_d^t(C_n)>\chi_d^t(W_n)      & \mbox{otherwise.}
\end{array}
\end{equation*}
\end{prop}

Propositions \ref{chi_d^t Pn, Cn} and \ref
{chi_d^t Wn, Cn} confirm the truth of the next remark.

\begin{rem}
\label{obs2} \emph{If $H$ is a subgraph of a graph $G$, we can not conclude that always $\chi_d^t(H)\leq \chi_d^t(G)$ holds or $\chi_d^t(H)\geq \chi_d^t(G)$}.
\end{rem}

\section{\bf Trees}
\vskip 0.4 true cm

In this section, we discuss on the total dominator chromatic number
of a \emph{tree}, which is a connected simple graph which has no
cycle. First we present some needed definitions. In a connected
graph $G$ the \emph{distance} between two vertices $u$ and $v$,
written $d_G(u,v)$ or simply $d(u,v)$, is the least length of a
$u$,$v$-path, and the \emph{diameter} of $G$, written
\emph{diam(G)}, is $\max_{u,v\in V(G)}d(u,v)$.

The \emph{eccentricity} of a vertex $u$, written $\epsilon(u)$, is
$\max_{v\in V(G)}d(u,v)$, while the \emph{radius} of $G$, written
\emph{rad(G)}, is $\min_{v\in V(G)}\epsilon(u)$. The \emph{center}
of $G$ is the subgraph induced by the vertices of minimum
eccentricity.

The following theorem describes the center of trees.

\begin{thm}
\label{Center of a tree} \emph{(\textbf{Jordan \cite{West}})}
 The
center of a tree is a vertex or an edge.
\end{thm}


In a tree, a \emph{leaf} is a vertex of degree one, while a
\emph{support vertex} is the neighbor of a leaf with degree more
than one. In this section, the set of leaves is denoted by $L$ and
$\ell=|L|$ , while the set of support vertices is denoted by $S$ and
$s=|S|$. I this section, we agree the following notations. Let
$S=\{v_i| 1 \leq i\leq s\}$, and $L=\{u_i | 1 \leq i\leq \ell\}$.
Also $\sigma$ denotes a function on $\{1,2,...,s\}$, the set of
indices of the elements of $S$, such that $\sigma(i)=j$ if $u_i$ is
adjacent to $v_j$. Hence $v_{\sigma(i)}$ denotes the support vertex
of $u_i$.

We start our discussion with the following lemma.

\begin{lem}
\label{chi_d^t(T)>=s+1} For any tree $T$ of order $n\geq 3$,
$\chi_d^t(T)\geq s+1$.
\end{lem}

\begin{proof}
$N({u_{i})=\{v_{\sigma(i)}}\}$ implies that in every TDC of $T$,
every vertex $v_i$ must be contained in a color class with
cardinality one. Since we must assign at least a new color to the
vertices in $L$, we obtain $\chi_d^t(T)\geq s+1$.
\end{proof}
Next proposition can be obtained easily and we have omitted its
proof.
\begin{prop}
\label{chi_d^t(T)=s+1, if V-L=S} Let $T$ be a tree of order $n\geq
3$. If every vertex in $T$ is a leaf or support vertex, then
$\chi_d^t(T)=s+1$.
\end{prop}


\begin{prop}
\label{chi_d^t(T)=s+1, if diam<=3} Let $T$ be a tree of order $n\geq
3$. If $diam(T)\leq 3$, then $\chi_d^t(T)=s+1$.
\end{prop}

\begin{proof}
$diam(T)\leq 3$ implies that for every two leaves $u_i$ and $u_j$,
there exist one of the $u_i$,$u_j$-paths:
$u_iv_{\sigma(i)}v_{\sigma(j)}u_j$ or $u_iv_{\sigma(i)}u_j$. Now
this fact that $(\{v_1\},\{v_2\},...,\{v_s\},V(T)-S)$ is a TDC of
$T$ and Lemma \ref{chi_d^t(T)>=s+1} imply $\chi_d^t(T)=s+1$.
\end{proof}

If we look carefully at the proof of Proposition
\ref{chi_d^t(T)=s+1, if diam<=3}, we may obtain next corollary.

\begin{cor}
\label{chi_d^t(T)>=s+2, if diam>=5} Let $T$ be a tree of order
$n\geq 3$ and $L\cup S\neq V(T)$. If $diam(T)\geq 5$, then
$\chi_d^t(T)\geq s+2$.
\end{cor}

\begin{prop}
\label{chi_d^t(T)=s+ 1 or 2, if diam=4} Let $T$ be a tree with
$diam(T)=4$. Then
\begin{equation*}
\chi_d^t(T)=\left\{
\begin{array}{ll}
s+1  & \mbox{if }d(u_i,u_j)=3, \mbox{ for some } u_i,u_j\in L, \\
s+2  & \mbox{otherwise.}
\end{array}
\right.
\end{equation*}
\end{prop}

\begin{proof}
$diam(G)=4$ implies the center of $T$ is a vertex, say $w$. If
$d(u_i,u_j)=3$, for some $u_i,u_j\in L$, then $\chi_d^t(T)=s+1$, by
Proposition \ref{chi_d^t(T)=s+1, if V-L=S}.

Now, assume $d(u_i,u_j)\neq 3$, for every $u_i,u_j\in L$. Then
$d(u_i,w)=2$ for any $u_i\in L$. Also for every two leaves $u_i$ and
$u_j$, there exist one of the $u_i$,$u_j$-paths:
$u_iv_{\sigma(i)}wv_{\sigma(j)}u_j$ or $u_iv_{\sigma(i)}u_j$. By the
contrary, let $\chi_d^t(T)=s+1$. Thus
$(\{v_1\},\{v_2\},...,\{v_s\},V(T)-S)$ is the only TDC of $T$. But
this is not possible, since for any $1\leq i\leq s$ vertex $v_i$ is
not adjacent to all vertices of a color class. Therefore,
$\chi_d^t(T)\geq s+2$. Now since
$(\{v_1\},\{v_2\},...,\{v_s\},\{w\},V(T)-(S\cup\{w\}))$ is a TDC of
$T$ with $s+2$ color classes, we obtain $\chi_d^t(T)=s+2$.
\end{proof}

\begin{prop}
\label{chi_d^t(T)=s+1,2,3, if diam=5} Let $T$ be a tree with
$diam(T)=5$ such that its center is edge $e_1e_2$. Then
\begin{equation*}
\chi_d^t(T)=\left\{
\begin{array}{ll}
s+1  & \mbox{if } e_1,e_2\in S, \\
s+2  & \mbox{if } |S|=2, \mbox{ or } |S\cap\{e_1,e_2\}|=1, \\
s+3  & \mbox{if } S\cap\{e_1,e_2\}=\emptyset, \mbox{ and } |S|\geq
3.
\end{array}
\right.
\end{equation*}
\end{prop}

\begin{proof}
Let $S=\{v_1,v_2\}$. Obviously $\chi_d^t(T)\geq 4$, and since
$(\{v_1\},\{v_2\},N(v_1),N(v_2))$ is a TDC of $T$ with cardinality
4, we obtain $\chi_d^t(T)=s+2$. Now we assume $|S|\geq 3$. If
$e_1,e_2\in S$, then $\chi_d^t(T)=s+1$, by Proposition
\ref{chi_d^t(T)=s+1, if V-L=S}. In the second case, we assume
$S\cap\{e_1,e_2\}=\{e_1\}$. By the contrary, let $\chi_d^t(T)=s+1$.
Thus $f=(\{v_1\},\{v_2\},...,\{v_s\},V(T)-S)$ is the only TDC of
$T$, and we must assign one color to the vertices in $L\cup\{e_2\}$.
But this implies that $f$ is not a TDC of $T$, a contrary. Therefore
$\chi_d^t(T)\geq s+2$, and since
$(\{v_1\},\{v_2\},...,\{v_s\},\{e_2\},L)$ is a TDC of $T$ with
cardinality $s+2$, we obtain $\chi_d^t(T)=s+2$.\\
Finally, let $S\cap\{e_1,e_2\}=\emptyset$. Then, obviously,
$\chi_d^t(T)\neq s+1$. If $\chi_d^t(T)=s+2$, then two new colors $i$
and $j$ must be assigned to the vertices in
$V(T)-S=L\cup\{e_1,e_2\}$ such that $e_1$ and $e_2$ have different
colors. Also we may assume $e_1\in N(v_1)$ and $e_2\in N(v_2)$.
Without loss of generality, we assign color $i$ to $e_1$ and color
$j$ to $e_2$. On the other hand, colors $i$ and $j$ can not be
assigned to the remained vertices, because $e_1\not\in N(v_2)$ and
$e_2\not\in N(v_1)$. Therefore, $\chi_d^t(T)\geq s+3$. Now since
$(\{v_1\},\{v_2\},...,\{v_s\},\{e_1\},\{e_2\},L)$ is a a TDC of $T$
with cardinality $s+3$, we obtain $\chi_d^t(T)=s+3$.
\end{proof}

\section{\bf Further research}
\vskip 0.4 true cm

We finish our discussion with some problems for further research.

\begin{prob}
Find $\chi_d ^t(T)$, when $T$ is a tree with diameter more than
five.
\end{prob}

\begin{prob}
Find some lower and upper bounds for $\chi_d ^t(G)+\chi_d
^t(\overline{G})$ and $\chi_d ^t(G)\cdot \chi_d ^t(\overline{G})$.
\end{prob}

\begin{prob}
For $k\geq 3$, characterize graphs $G$ satisfy $\chi_d ^t(G)=k$.
\end{prob}

\begin{prob}
Characterize graphs $G$ satisfy\\
$\bullet$ $\chi_d ^t(G)=\chi_d(G)$,\\
$\bullet$ $\chi_d ^t(G)=\chi(G)$,\\
$\bullet$ $\chi_d ^t(G)=\gamma_t(G)$, or\\
$\bullet$ $\chi_d^t(G) = \gamma_t(G)+ \min_S \chi(G[V(G)-S])$, where
$S\subset V(G)$ is a $\gamma_t(G)$-set.
\end{prob}


\bigskip
\bigskip




\begin{thebibliography}{20}


\bibitem{MM} M. Chellali, F. Maffray, Dominator Colorings in Some Classes of Graphs,
{\em Graphs and Combinatorics}, {\bf 28} (2012) 97--107.

\bibitem{GJ} M. R. Garey, D. S. Johnson, Computers and Intractability, {\em W. H. Freeman and Co.}, 1978.

\bibitem{Ger1} R. Gera, On the dominator colorings in bipartite graphs, {\em Inform. Technol. NewGen.}, {\bf ITNG'07} (2007) 947--952.

\bibitem{Ger2} R. Gera,  On dominator colorings in graphs, {\em Graph Theory Notes}, {\bf N. Y. LII} (2007) 25--30.

\bibitem{GHR} R. Gera, S. Horton, C. Rasmussen, Dominator colorings and safe clique partitions, {\em Congress.
Num.}, {\bf 181} (2006) 19--32.


\bibitem{hhs1} T. W. Haynes, S. T. Hedetniemi, P. J. Slater (Eds.),
{\em Fundamentals of Domination in Graphs}, Marcel Dekker, Inc. New
York, 1998.

\bibitem{hhs2} T. W. Haynes, S. T. Hedetniemi, P. J. Slater (Eds.),
{\em Domination in Graphs: Advanced Topics}, Marcel Dekker, Inc. New
York, 1998.


\bibitem{Kaz.Pah} A. P. Kazemi, B. Pahlavsay, $k$-tuple total
domination number of Cartesian product graphs, {\em submitted}.









\bibitem{West} D. B. West, \emph{Introduction to Graph Theorey}, 2nd ed., Prentice Hall, USA, 2001.

\end{thebibliography}
\end{document}